\documentclass{article}

\usepackage{fullpage}
\usepackage{amsmath,amssymb,enumerate}
\usepackage{latexsym,amsthm}
\usepackage[pdftex]{graphicx}
\usepackage{color}
\usepackage[pdftex,bookmarks,colorlinks]{hyperref}

\hypersetup{%
  ocgcolorlinks,
  colorlinks=true,%
  linkcolor=[rgb]{0.5,0.0,0.0},%
  citecolor=[rgb]{0,0,0.45}
}

\newtheorem{theorem}{Theorem}
\newtheorem{lemma}[theorem]{Lemma}

\newtheorem{conjecture}{Conjecture}
\newtheorem*{dilemma}{Distance Lemma}

\newenvironment{denseitems}{\list{$\bullet$}%
  {\labelwidth3em\itemsep0pt\parsep0pt\topsep0.6ex}}{\endlist}

\newcommand{\R}{{\mathbb{R}}}
\newcommand{\N}{{\mathbb{N}}}
\newcommand{\s}{^\ast}
\newcommand{\Sp}{\mathfrak{S}}
\newcommand{\F}{{\mathfrak{F}}}
\newcommand{\G}{{\mathfrak{G}}}
\newcommand{\K}{{\mathfrak{K}}}
\newcommand{\C}{{\mathcal{C}}}
\newcommand{\cone}{\mathfrak{C}}
\newcommand{\ello}{\ell^{\perp}}
\newcommand{\sigmao}{\sigma^{\perp}}
\newcommand{\pth}[1]{\left( #1 \right)}
\newcommand{\eps}{\varepsilon}
\newcommand{\vc}[1]{\overrightarrow{#1}}
\def\p #1.{\mathit{#1}}

\let\geq\geqslant
\let\leq\leqslant
\let\ge\geqslant
\let\le\leqslant

\makeatletter
\ifx\showkeyslabelformat\@undefined\else
\renewcommand{\showkeyslabelformat}[1]{\normalfont\tiny\ttfamily#1}
\fi
\partopsep\z@ \textfloatsep 10pt plus 1pt minus 4pt
\def\section{\@startsection {section}{1}{\z@}{-3.5ex plus -1ex minus
    -.2ex}{2.3ex plus .2ex}{\large\bf}}
\def\subsection{\@startsection{subsection}{2}{\z@}{-3.25ex plus -1ex
    minus -.2ex}{1.5ex plus .2ex}{\normalsize\bf}}
\def\@fnsymbol#1{\ensuremath{\ifcase#1\or *\or 1\or 2\or 3\or 4\or
    5\or 6\or 7 \or 8\ or 9 \or 10\or 11 \else\@ctrerr\fi}}
\makeatother

\title{Geometric Permutations of Non-Overlapping Unit Balls Revisited%
  \thanks{This research was supported in part by NRF grant
    2011-0030044 (SRC-GAIA), and in part by NRF grant~2011-0016434,
    both funded by the government of Korea.}}

\author{Jae-Soon Ha%
  \thanks{KAIST, Korea, Email: jaesoonha@kaist.ac.kr,
    otfried@kaist.edu}
  \and
  Otfried Cheong\footnotemark[2]
  \and
  Xavier Goaoc%
  \thanks{Universit\'{e} Paris-Est Marne-la-Vall\'{e}e, France, Email:
    goaoc@univ-mlv.fr}
  \and
  Jungwoo Yang%
  \thanks{Aarhus University, Denmark, Email: jungwoo@madalgo.au.dk}
}

\begin{document}
\maketitle
\begin{abstract}
  Given four congruent balls $A, B, C, D$ in~$\R^{d}$ that have
  disjoint interior and admit a line that intersects them in the
  order~$\p ABCD.$, we show that the distance between the centers of
  consecutive balls is smaller than the distance between the centers
  of~$A$ and~$D$.  This allows us to give a new short proof that $n$
  interior-disjoint congruent balls admit at most three geometric
  permutations, two if $n\ge 7$. We also make a conjecture that would
  imply that $n\geq 4$ such balls admit at most two geometric
  permutations, and show that if the conjecture is false, then there
  is a counter-example of a highly degenerate nature.
\end{abstract}

\section{Introduction}
A \emph{line transversal} to a family $\F$ of pairwise disjoint convex
sets in~ $\R^d$ is a line that intersects every element of that
family. The study of line transversals, their properties, and
conditions for their existence started in the 1950s with the classic
work of Gr\"unbaum, Hadwiger, and Danzer; background about the
sizable literature on \emph{geometric transversal theory} can be
found in the classic survey of Danzer et al.~\cite{dgk-htr-63}, or the
more recent ones by Goodman et al.~\cite{gpw-gtt-93},
Eckhoff~\cite{e-hrctt-93}, Wenger~\cite{wenger1990geometric}, or
Holmsen~\cite{holmsen2008recent}.

An oriented line transversal~$\ell$ to a family~$\F$ induces a linear
order on~$\F$: Fig.~\ref{fig:intro1}(a) shows three oriented transversals
to a family of three congruent disks inducing the orders $A \prec C
\prec B$, $A \prec B \prec C$, and $B \prec A \prec C$.
\begin{figure}[h]
  \centerline{\includegraphics{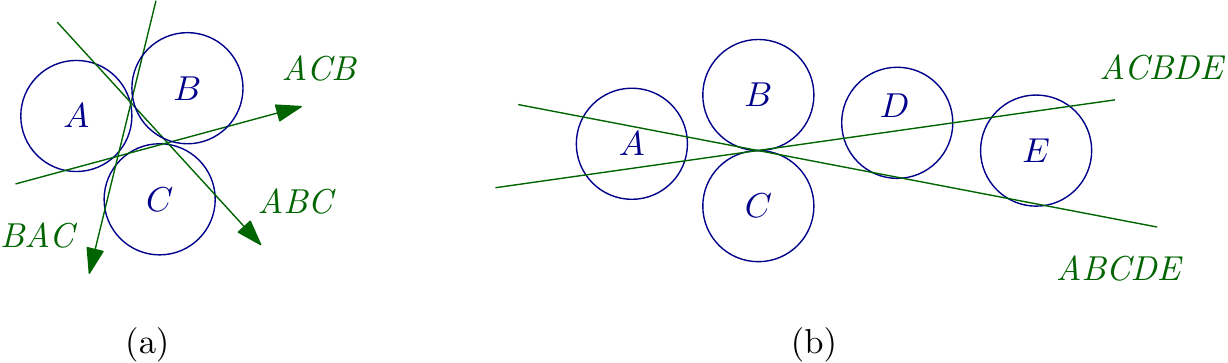}}
  \caption{Orders and geometric permutations\label{fig:intro1}}
\end{figure}
For conciseness, we usually represent the order by the string listing
the elements, the three lines in Fig.~\ref{fig:intro1}(a) induce the
orders $\p ACB.$, $\p ABC.$, and $\p BAC.$. Natural questions in
geometric transversal theory are: Given a family of disjoint convex
objects, how many different orders can be realized by line
transversals?  How much can these orders differ? What becomes of these
questions if the objects have a more restricted shape, for instance if
they are balls or axis-aligned boxes?

If an order can be realized by an oriented line, so can its reverse,
so the two are equivalent in this sense.  The equivalence classes,
that is, pairs of an order and its reverse, are called \emph{geometric
  permutations}.  Fig.~\ref{fig:intro1}(b) shows a set of five
congruent disks with the two geometric permutations~$\p ABCDE.$ and
$\p ACBDE.$, which could equally well be written as~$\p EDCBA.$ and
$\p EDBCA.$. In Fig.~\ref{fig:intro1}(b) the disks~$B$ and~$C$ touch
each other. We allow this, but a line transversal is not allowed to be
tangent to these disks in this common point.  Put differently, we can
remove the common points of contact from the objects to obtain a
family of disjoint convex objects with the same set of line
transversals.  It is convenient to allow such families, as
configurations are often easier to describe when objects touch.  We
will call a family of compact convex objects in~$\R^{d}$ that may
touch, but whose interior is disjoint, a \emph{non-overlapping}
family.

The study of geometric permutations started in the 1980s with the work
by Katchalski et al.~\cite{katchalski1985geometric,
  katchalski1987geometric}. In the plane, $n$ convex objects admit at
most $2n-2$ geometric permutations and this bound is
tight~\cite{edelsbrunner1990maximum}.  One of the intriguing open
questions is the corresponding bound for three and higher dimensions:
$n$ convex objects in~$\R^{d}$ can have $\Omega(n^{d-1})$ geometric
permutations~\cite{smorodinsky2000sharp}, but the best known upper
bound is only $O(n^{2d-3}\log n)$~\cite{Rubin2012}. For balls or similar
\emph{fat} objects, the lower bound of $\Omega(n^{d-1})$ is known to
be tight~\cite{smorodinsky2000sharp,katz2001tight}. Disjoint
\emph{congruent} balls, however, have only a constant number of
geometric permutations: In two dimensions, $n \geq 4$ congruent disks
have at most two geometric
permutations~\cite{smorodinsky2000sharp,asinowski2003geometric}.  In
dimension~$d \geq 3$, Cheong et al.~\cite{cheong2005geometric} proved
that $n$ non-overlapping congruent balls have at most three geometric
permutations, and at most two geometric permutations when~$n\geq 9$.

\bigskip

In this paper we revisit the problem of bounding the number of
geometric permutations of $n$~non-overlapping congruent balls
in~$\R^d$.  Since we can arbitrarily choose the radius of the balls,
we will refer to them as \emph{unit balls}.  The earlier work of
Cheong et al.~\cite{cheong2005geometric} does not entirely settle the
question, as no construction of $n > 3$ non-overlapping unit balls is
known that admits more than two geometric permutations.  Furthermore,
the proof by Cheong et al.\ is quite technical and relies on delicate
geometric lemmas and tedious case analysis.

\bigskip

In the first part of this paper, we give a shorter and greatly
simplified proof that $n \geq 3$ non-overlapping unit balls have at
most three geometric permutations.  Unlike the previous
proof~\cite{cheong2005geometric}, it could be presented in its
entirety in an undergraduate course on transversal theory. Our main
theorem is the following:
\begin{theorem}
  \label{the:main}
  Let $\F$ be a family of $n$ non-overlapping unit balls in $\R^d$.
  The number of geometric permutations of $\F$ is at most three if $n
  \le 6$, and at most two if $n \ge 7$.
\end{theorem}

\noindent
Theorem~\ref{the:main} slightly improves the previous bound of Cheong
et al.~\cite{cheong2005geometric} by settling the question for $n=7$
and $8$. Our proof rests on the following lemma:

\begin{dilemma}
  If four non-overlapping unit balls $A$, $B$, $C$ and $D$ in $\R^d$
  have a line transversal with the order $ABCD$ then $|ad| >
  \max\{|ab|,|bc|,|cd|\}$.
\end{dilemma}
\noindent
(Here and throughout the paper we will use lower-case letters to
denote the centers of balls written with upper-case letters, so $a$,
$b$, $c$, and $d$ are the centers of $A$, $B$, $C$, and~$D$.)  The
lemma is not as obvious as it might appear: For three unit-balls, for
instance, the existence of a transversal with the order $\p ABC.$ does
not imply that $|ac| > |ab|$, as already evidenced by
Fig.~\ref{fig:intro1}(a).

We prove the distance lemma, in Section~\ref{sec:distance-lemma}, by
first modifying the given configuration into a canonical situation: We
shrink the balls, keeping them congruent, until we reach the smallest
radius for which they still have a transversal with the given order.
This idea has probably been used first by Klee~\cite{klee1954common}
and then by Hadwiger~\cite{hadwiger1957eibereiche}.  The resulting
canonical configuration~$\F$ has the property that the line
transversal~$\ell$ is \emph{pinned} (Lemma~\ref{lem:shrinking}): This
means that any arbitrarily small perturbation of~$\ell$ is no longer a
transversal of~$\F$.  In other words, $\ell$ is an isolated point in
the space of transversals of~$\F$.  The same method for deforming a
family of unit balls such that the line transversal becomes pinned has
been used by Cheong et al.~\cite{cheong2008helly}.  The correctness of
the method is there deduced from algebraic results by Megyesi and
Sottile~\cite{megyesi05} and by Borcea et al.~\cite{borcea2008line}.
This argument requires strict disjointness of the balls, and doesn't
meet our goal of a proof presentable to undergraduates.  We instead
observe that the fact we need is already implicit in a proof by
Holmsen et al.~\cite{holmsen2003helly}.  In Appendix~\ref{sec:holmsen}
we examine their proof to prove the correctness of the pinning method
for non-overlapping unit balls.

Before proving the distance lemma, we show, in
Section~\ref{sec:at-most-three}, that it readily simplifies various
steps of the proof of Cheong et al.~\cite{cheong2005geometric},
resulting in an elementary proof that the number of geometric
permutations of $n$~non-overlapping unit balls is at most three. On
the one hand, the distance lemma simplifies technical derivations.
For example, the fact that the geometric permutations $\p ABCD.$
and~$\p BADC.$ are \emph{incompatible} for non-overlapping unit balls,
that is, they cannot be realized at the same time by a family of four
balls, was given a delicate, five pages long,
proof~\cite[Section~4]{cheong2005geometric}; it follows immediately
from the distance lemma, since $\p ABCD.$ implies that $|ad| > |bc|$
and $\p BADC.$ implies that $|bc| > |ad|$, a contradiction. On the
other hand, using the distance lemma we can replace rather pedestrian
arguments by more conceptual analyses, for instance the mechanical
reduction from $n$ to $4$ balls~\cite[Section~2]{cheong2005geometric}
is done more concisely in Lemma~\ref{lem:ThreeFourGP}.

\bigskip

We conjecture that the geometric permutations $\p ABCD.$ and $\p
ACDB.$ are incompatible.  If proven, this would show that $n\ge 4$
non-overlapping unit balls have at most two geometric permutations,
thereby completely closing this question. In the second part of this
paper, we analyze the geometry of certain pinning configurations and
show that if our conjecture is false then it must admit
counter-examples of a highly contrived nature.

\section{At most three geometric permutations}
\label{sec:at-most-three}

We first use the distance lemma to reduce the problem from $n$~balls
to three or four balls (the same result was obtained by Cheong et
al.~\cite{cheong2005geometric} via a tedious case-analysis):
\begin{lemma}
  \label{lem:ThreeFourGP}
  If $n \geq 4$ non-overlapping unit balls in $\R^d$ have at least $k
  \in \{3,4\}$ geometric permutations, then $k$ of the balls have $k$
  distinct geometric permutations.
\end{lemma}
\begin{proof}
  Let $\F$ be a family of $n \geq 4$ non-overlapping unit balls in
  $\R^d$. We call an element \emph{extreme} in a geometric permutation
  if it appears first or last in its order. We make two observations:
  \begin{enumerate}
  \item[(i)] Any two geometric permutations of $\F$ have an extreme
    element in common. Indeed, if two geometric permutations $\sigma_1$
    and $\sigma_2$ of $\F$ have disjoint sets of extreme elements
    $\{A_1,B_1\}$ and $\{A_2,B_2\}$ then applying the distance lemma to
    $\sigma_1$ yields $|a_1b_1|>|a_2b_2|$ and applying it to $\sigma_2$
    yields $|a_2b_2|>|a_1b_1|$, a contradiction.

  \item[(ii)] If two geometric permutations of $\F$ share an extreme
    element $A$ then they differ on $\F \setminus \{A\}$. Indeed, assume
    that the first geometric permutations writes $\p AB\ldots XY.$. If the
    second, which is distinct from the first, coincides with it on $\F
    \setminus \{A\}$ then it must be $\p AYX\ldots B.$. The distance lemma
    then implies both that $|ay|>|ab|$ and that $|ab|>|ay|$, a
    contradiction.
  \end{enumerate}
  Assume that $\F$ has three geometric permutations and let $\G$ be a
  minimal subfamily of $\F$ on which their restrictions~$\tau_1,
  \tau_2$, and $\tau_3$ are pairwise distinct.  There cannot be an
  extreme element common to all three~$\tau_i$ as observation~(ii)
  would contradict the minimality of~$\G$.  Hence, there exist three
  distinct elements $A,B,C \in \G$ such that $A$ is extreme in
  $\tau_1$ and $\tau_2$, $B$ is extreme in $\tau_2$ and $\tau_3$ and
  $C$ is extreme in $\tau_1$ and $\tau_3$. Then the restrictions of
  $\tau_1$, $\tau_2$ and $\tau_3$ to $\{A,B,C\}$ are $\p ABC.$, $\p
  ACB.$ and $\p BAC.$, implying $\G = \{A, B, C\}$.

  Assume now that $\F$ has four geometric permutations and, again, let
  $\G$ be a minimal subfamily of~$\F$ on which their restrictions
  $\tau_1, \tau_2, \tau_3$, and $\tau_4$ are pairwise distinct.  For $H
  \subseteq \G$ let ${\tau_i}_{|H}$ denote the restriction of
  $\sigma_i$ to $H$. As we just argued, $\G$ contains a triple $T =
  \{A,B,C\}$ such that $A$ is extreme in $\tau_1$ and $\tau_2$, $B$ is
  extreme in $\tau_2$ and $\tau_3$ and $C$ is extreme in $\tau_1$ and
  $\tau_3$; we further have
  \[
    {\tau_1}_{|T} = \p ABC., \quad {\tau_2}_{|T} = \p ACB., \quad {\tau_3}_{|T}
    = \p BAC..
  \]
  By observation~(i) the extreme elements of $\tau_4$ are among
  $\{A,B,C\}$, say $A$ and $C$. Since $\tau_1$ and $\tau_4$ have the
  same extreme elements but are different there must exist a pair
  $\{D,E\} \subset \G$ such that the restrictions of $\tau_1$ and
  $\tau_4$ to $\{A,C,D,E\}$ are different. Assume that $B \notin
  \{D,E\}$ so that $Q=\{A,B,C,D,E\}$ has size five. We write
  ${\tau_1}_{|Q} = AX_1X_2X_3C$, ${\tau_2}_{|Q} = AY_1Y_2Y_3B$ and
  ${\tau_3}_{|Q} = BZ_1Z_2Z_3C$.  If $X_3=B$ then $|bc|<|ac|$ and we
  must set $Z_1=A$. This implies that $|ab|<|bc|$ and we must set
  $Y_1=C$. This implies that $|ac|<|ab|$ which contradicts the two
  previous inequalities. It must then be that $X_3 \neq B$. This
  implies that $|ab|<|ac|$ which forces $Y_3=C$. This implies
  $|bc|<|ab|$, which forces $Z_3=A$. This implies that $|ac|<|bc|$,
  again a contradiction with the two previous inequalities. As a
  consequence, $B \in \{D,E\}$ and $\tau_1, \ldots \tau_4$ are already
  distinct on the quadruple $\{A,C,D,E\}$.
\end{proof}

We can now easily prove that there cannot be more than three geometric
permutations.
\begin{theorem}
  \label{the:at-most-three}
  A family of non-overlapping unit balls in $\R^d$ has at most three
  geometric permutations.
\end{theorem}
\begin{proof}
  By Lemma~\ref{lem:ThreeFourGP} it suffices to prove the statement for
  families of size four. Let $A,B,C,D$ be four non-overlapping unit
  balls in $\R^d$ and assume that there is a line transversal in the
  order~$\p ABCD.$. The distance lemma implies that $|ad| >
  \max\{|ab|,|bc|,|cd|\}$ and no line can meet these balls in the
  order $\p ADCB.$ (which implies $|ab|>|ad|$), $\p BADC.$ (which
  implies $|bc|>|ad|$), $\p BDAC.$ (implying $|bc|>|ad|$), or $\p
  CBAD.$ (as this entails $|cd|>|ad|$). Of the twelve geometric
  permutations of four elements, this leaves the seven shown in
  Fig.~\ref{fig:incomp7gp} as candidates for the remaining geometric
  permutations of~$\F$.
  \begin{figure}[h]
    \centerline{\includegraphics{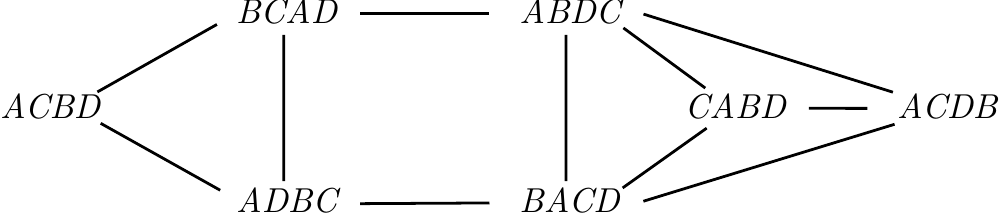}}
    \caption{Proof of Theorem~\ref{the:at-most-three}.}
    \label{fig:incomp7gp}
  \end{figure}
  It is easy to verify that the geometric permutations connected by
  edges in Fig.~\ref{fig:incomp7gp} are also incompatible by the
  distance lemma.  The resulting graph has no independent set of size
  larger than two, and so $\F = \{A, B, C, D\}$ has at most three
  geometric permutations.
\end{proof}

To prove the stronger statement of Theorem~\ref{the:main}, we need two
lemmas proven by Cheong et al.~\cite{cheong2005geometric} (their
proofs are short and self-contained).  For a directed line~$\ell$, we
write $\vec{\ell}$ for its direction vector, for points $p, q \in
\R^{d}$, we will write $\vc{pq}$ for the vector $q - p$ from~$p$
to~$q$.
\begin{lemma}[{\cite[Lemma~7]{cheong2005geometric}}]
  \label{lem:angle}
  Given three non-overlapping unit spheres $A$, $B$ and $C$ in
  $\R^{d}$, and a directed line~$\ell$ stabbing them in the order $\p
  ABC.$. Then $\angle(\vec{\ell}, \vc{ac}) < \pi/4$.
\end{lemma}
\begin{lemma}[{\cite[Lemma~6]{cheong2005geometric}}]
  \label{lem:cylinder}
  Let $\C$ be a cylinder of radius one and length less than
  $s\sqrt{2}$ in~$\R^{d}$, for some $s \in \N$. Then $\C$ contains at
  most $2s$ points with pairwise distance at least two.
\end{lemma}
We analyze the intersection of two cylinders more carefully in the
following lemma:
\begin{lemma}
  \label{lem:cylinder6}
  Let $\C_1$ and~$\C_2$ be cylinders of radius one and axes~$\sigma_1$
  and~$\sigma_2$ in~$\R^{d}$.  If $\pi/4 < \angle(\vc{\sigma_1},
  \vc{\sigma_2}) \leq \pi/2$, then the intersection~$\C_1 \cap \C_2$
  contains at most six points with pairwise distance at least two.
\end{lemma}
\begin{proof}
  We choose a coordinate system where $\sigma_1$ is the $x_1$-axis,
  and $\sigma_2$ is the line $(t\cos\theta, t\sin\theta, d, 0, \ldots,
  0)$, where $\theta = \angle(\vc{\sigma_1}, \vc{\sigma_2}) > \pi/4$
  and $d \geq 0$ is the distance between~$\sigma_1$ and~$\sigma_2$.
  The left side of Fig.~\ref{fig:cylinders} shows the projection on
  the $x_1 x_2$-plane.
  \begin{figure}[h]
    \centerline{\includegraphics{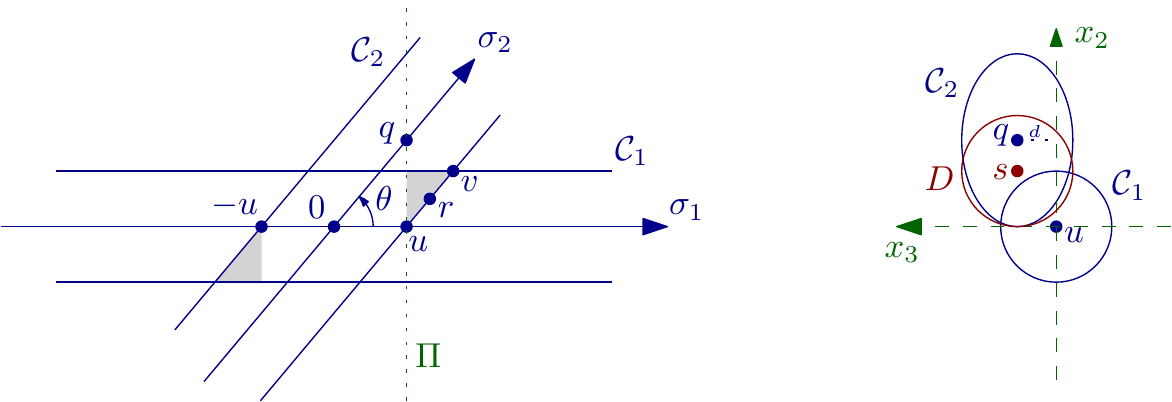}}
    \caption{Projection of $\C_1$ and $\C_2$ on the $x_1 x_2$-plane.}
    \label{fig:cylinders}
  \end{figure}
  Consider the points $u = (1/\sin\theta, 0, 0, \dots, 0)$ and $v =
  (1/\sin\theta + \cot\theta, 1, 0, \dots, 0)$ marked in the figure.
  Since $\theta > \pi/4$, the distance between $u$ and~$-u$ is less
  than $2\sqrt{2}$, and so by Lemma~\ref{lem:cylinder} the section of
  $\C_1$ between $-u$ and~$u$ contains at most four points of pairwise
  distance at least two.  All remaining points in $\C_1 \cap \C_2$
  must project into the two symmetric shaded regions.  We will now
  show that these regions have diameter less than two and can
  therefore contain only one point each, proving the lemma.

  Let $p \in \C_1 \cap \C_2$ be a point with $u_1 < p_1 \leq v_1$ (we
  use indices for the coordinates in~$\R^{d}$).  We will show that
  $|pr| < 1$, where $r = \big((u_1 + v_1)/2, \frac 12, \frac d2, 0,
  \dots, 0\big)$ (note that $r$ does not lie in the $x_1 x_2$-plane).
  Since $\theta > \pi/4$, we have $v_1 - u_1 = \cot\theta < 1$, and so
  $|p_1 - r_1| < 1/2$.  Let $\Pi$ denote the hyperplane $x_1 = u_1$,
  and let $p\s$ and $r\s$ be the orthogonal projection of $p$ and~$r$
  into~$\Pi$.  We observe that $p\s \in \C_1 \cap \C_2$. The
  intersection $\C_1 \cap \Pi$ is the unit-radius ball around the
  origin in~$\Pi$.  The intersection $\C_2 \cap \Pi$ is an ellipsoid
  with center~$q = (1/\sin\theta, 1/\cos\theta, d, 0, \dots, 0)$, see
  right hand side of Fig.~\ref{fig:cylinders}.  $\C_2 \cap \Pi$ contains exactly
  the points $x = (u_1, x_2, x_3, \dots, x_d) \in \Pi$ with
  \[
  (\cos\theta)^{2} (x_2-\frac{1}{\cos\theta})^{2} + (x_3 - d)^{2}
  + \sum_{i = 4}^{d} x_{i}^{2} \leq 1.
  \]
  Consider the ball~$D$ with center $s = (u_1, 1, d, 0, \dots, 0)$ and
  radius one.  For a point $x \in \Pi$ with $0 \leq x_2 \leq 1$, we
  have
  \[
  1 - x_2 \leq 1 - (\cos\theta)x_2 = (\cos
  \theta)(\frac{1}{\cos\theta} - x_2),
  \]
  and so $x \in \C_2$ implies $x \in D$.  It follows that $p\s \in
  \C_1 \cap D$. Since $|us| \geq 1$ and $r\s = (u + s)/2$ is the
  midpoint of the two centers we have $|p\s r\s| \leq \sqrt{3}/2$. It
  follows that $|pr|^{2} = |p_1 - r_1|^{2} + |p\s r\s|^{2} < 1/4 + 3/4
  = 1$.
\end{proof}

\begin{proof}[Proof of Theorem~\ref{the:main}]
  We proved the bound for $n \leq 6$ in
  Theorem~\ref{the:at-most-three}, so it remains to consider
  families~$\F$ of $n \geq 7$ balls.  We show that the geometric
  permutations $\p XYZU.$ and $\p XUYZ.$ are incompatible for~$\F$.
  Assume for a contradiction that $\ell$ is an oriented transversal
  inducing the order $\p XYZU.$, and $\ell'$ is an oriented
  transversal inducing the order $\p XUYZ.$.  By
  Lemma~\ref{lem:angle}, we have $\angle(\vec{\ell}, \vc{\ell'}) \leq
  \angle(\vec{\ell}, \vc{xz}) + \angle(\vc{\ell'}, \vc{xz}) <
  \pi/2$.  Since $\ell'$ meets $U$ before $Y$, we have
  $\angle(\vc{\ell'}, \vc{yu}) > \pi/2$, and by
  Lemma~\ref{lem:angle} again we have $\angle(\vec{\ell}, \vc{yu}) <
  \pi/4$, implying $\angle(\vec{\ell}, \vc{\ell'}) > \pi/4$.
  Consider now the cylinders~$\C$ and~$\C'$ of radius one with
  axes~$\ell$ and~$\ell'$.  Since $\ell$ and $\ell'$ are
  transversals for~$\F$, the centers of all balls in~$\F$
  are contained in $\C \cap \C'$. By Lemma~\ref{lem:cylinder6}, this
  implies $n \leq 6$, a contradiction.

  We now assume that $\F$ has three geometric permutations. By
  Lemma~\ref{lem:ThreeFourGP} there is a subset $\G = \{A, B, C, D\}$
  of four balls such that $\G$ already has three geometric
  permutations. We can assume $\p ABCD.$ is one of them.  The
  incompatible pair $(\p XYZU., \p XUYZ.)$ implies that $\p ACDB.$,
  $\p ADBC.$, $\p CABD.$, and $\p BCAD.$ cannot exist.  Of the
  geometric permutations shown in Fig.~\ref{fig:incomp7gp}, this only
  leaves $\p ACBD.$, $\p ABDC.$, and $\p BACD.$. Since we already know
  $\p ABDC.$ and $\p BACD.$ to be incompatible by the distance lemma
  (see Fig.~\ref{fig:incomp7gp}), the last pair must therefore include
  $\p ACBD.$.  But this permutation is incompatible with the other two
  because they form pairs of the form~$(\p XYZU., \p XUYZ.)$.
\end{proof}

\section{Proof of the distance lemma}
\label{sec:distance-lemma}

We say that a family $\F$ \emph{pins} a line $\ell$ or that $\ell$ is
pinned by $\F$ if $\ell$ is a line transversal to $\F$ and any
arbitrarily small perturbation of $\ell$ is not a line transversal to
$\F$. \footnote{Equivalently, a line is pinned by $\F$ if it is an
  isolated point in the space of line transversals to $\F$ endowed
  with the natural topology on the space of lines, for instance as
  given by the Grassmann-Pl\"ucker coordinates.} It is often
convenient to deform a family of balls and lines into a configuration
where the lines are pinned.  The following lemma describes such a
deformation.
\begin{lemma}
  \label{lem:shrinking}
  Let $\F(t)=\{B_1(t),B_2(t), \ldots, B_n(t)\}$ be a parameterized
  family of non-overlapping balls of radius~$t \in [0,1]$ in $\R^3$,
  with the property that $B_{i}(s) \subset B_{i}(t)$ for any $1 \leq i
  \leq n$ and $0 \leq s < t \leq 1$. If $\F(1)$ has a line transversal
  in the order $B_1(1)B_2(1)\ldots B_n(1)$ then there exists $t\s \in
  [0,1]$ such that $\F(t\s)$ has a pinned line transversal in the
  order $B_1(t\s)B_2(t\s)\ldots B_n(t\s)$.
\end{lemma}

\noindent
The proof of Lemma~\ref{lem:shrinking} is already implicit in Holmsen
et al.~\cite{holmsen2003helly}.  For completeness, we revisit their
proof in Appendix~\ref{sec:holmsen} and make the necessary
adjustments.

The following lemma allows us to reduce the dimension in which we have
to prove our statements.
\begin{lemma}
  \label{lem:dim-red}
  Let $\F$ be a family of non-overlapping unit balls in $\R^d$ with a
  line transversal~$\ell$, and let $\Sp$ be an affine subspace
  containing the centers of all balls in~$\F$.  The orthogonal
  projection~$\ell'$ of $\ell$ into the subspace~$\Sp$ is a line
  transversal to $\F$ realizing the same geometric permutation.
\end{lemma}
\begin{proof}
  If $p$ is a point in a ball $B \in \F$ and $p'$ is the projection of
  $p$ into~$\Sp$, then $|bp'| \leq |bp|$, and so $p' \in B$. The lemma
  follows.
\end{proof}

We will use the following folklore characterization of triples of
balls pinning a line (we include a proof for completeness).
\begin{lemma}
  \label{lem:tri-tang}
  A set $\{A,B,C\}$ of three non-overlapping unit balls in $\R^3$ pins
  a line $\ell$ if and only if they are tangent to $\ell$, their
  centers $a,b,c$ are coplanar with $\ell$, and in that plane $\ell$
  separates the center of the middle ball (in the order of tangency)
  from the other two centers.
\end{lemma}
\begin{proof}
  If $\{A,B,C\}$ satisfies the condition then the ball $B$ is
  separated from $A$ and $C$ by the plane $\Pi$ perpendicular in
  $\ell$ to the plane of centers. Any line transversal to $\{A,B,C\}$
  in the same order as $\ell$ must be contained in $\Pi$, and thus
  $\ell$ is pinned.

  Conversely, assume that $\{A,B,C\}$ pins $\ell$.  If $\ell$ is not
  contained in the plane of centers, then rotating $\ell$ toward its
  orthogonal projection into that plane decreases the distances to all
  centers as in Lemma~\ref{lem:dim-red}.  Since any intermediate line
  in this rotation is therefore a line transversal, $\ell$ is not
  pinned. The line $\ell$ is thus contained in the plane of centers, and
  the necessity of the separation condition is easily checked.
\end{proof}

\begin{proof}[Proof of distance lemma.]
  We first argue that the statement follows from the case~$d=3$.
  Indeed, let $\Sp$ denote a $3$-dimensional space containing the four
  balls' centers (if they are not coplanar, then $\Sp$ is uniquely
  defined). The space $\Sp$ intersects the four non-overlapping
  $d$-dimensional unit balls in four non-overlapping $3$-dimensional
  unit balls with the same centers. Let $\ell'$ denote the orthogonal
  projection of $\ell$ into $\Sp$. By Lemma~\ref{lem:dim-red}, $\ell'$ is
  a line transversal to the four $3$-dimensional balls with the
  ``same'' geometric permutation as~$\ell$.  It therefore suffices to
  prove the claim for the $3$-dimensional balls and~$\ell'$.

  We now assume that we are in the case~$d=3$.  We shrink the balls
  uniformly around their center. By Lemma~\ref{lem:shrinking} we will
  reach a configuration with transversal $\ell$ that is pinned by four
  non-overlapping unit balls~$\{A,B,C,D\}$. If the four centers are
  collinear then the statement is clear, so we assume otherwise.
  \begin{figure}[ht]
    \centerline{\includegraphics{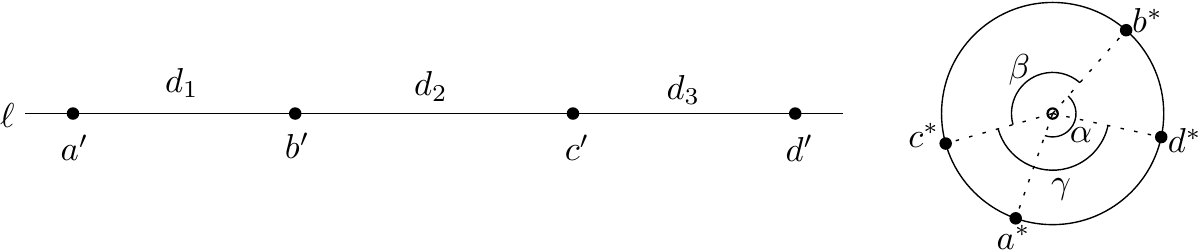}}
    \caption{Notation for the proof of the distance lemma (here all
      four balls are shown tangent to the pinned line).}
    \label{fig:LemmaDistanceABCD}
  \end{figure}

  We will use the following notation (refer to
  Fig.~\ref{fig:LemmaDistanceABCD}). Let $\ello$ denote a plane
  orthogonal to~$\ell$. For $x \in \{a,b,c,d\}$, let $x'$ denote the
  orthogonal projection of $x$ onto $\ell$ (that is, the point closest
  to $x$ on~$\ell$), and let $x\s$ denote the orthogonal projection
  of $x$ onto~$\ello$.  We set $d_1=|a'b'|$, $d_2=|b'c'|$,
  $d_3=|c'd'|$, and $\Delta = 2(d_1d_2 +d_1d_3 + d_2d_3)$.
  We have:
  \[
  |ad|^2 = (d_1 + d_2 + d_3)^2 + |a\s d\s|^{2}, \quad
  |ab|^2 = d_{1}^{\,2}+|a\s b\s|^2, \quad
  |bc|^2 = d_{2}^{\,2}+|b\s c\s|^2, \quad
  |cd|^2 = d_{3}^{\,2}+|c\s d\s|^2.
  \]
  Since $\ell$ is a line transversal with order~$\p ABCD.$, we have
  $d_{1}, d_{2}, d_3 > 0$ and therefore $\Delta > 0$.

  Our goal is to prove that
  \begin{align}
    \label{i:c1}
    |ad|^2 - |ab|^2 & =
    \Delta + d_2^{\,2}+d_3^{\,2} +|a\s d\s|^{2}-|a\s b\s|^2 > 0 \\
    \label{i:c2}
    |ad|^2 - |bc|^2 & =
    \Delta + d_1^{\,2}+d_3^{\,2} +|a\s d\s|^{2}-|b\s c\s|^2 > 0 \\
    \label{i:c3}
    |ad|^2 - |cd|^2 & =
    \Delta + d_1^{\,2}+d_2^{\,2} +|a\s d\s|^{2}-|c\s d\s|^2 > 0
  \end{align}

  Assume first that three of the balls already pin $\ell$.  There are
  essentially two cases:
  \begin{itemize}
  \item If the three balls are $\{A,B,C\}$ then, by
    Lemma~\ref{lem:tri-tang}, $a\s = c\s$ and $|b\s c\s|^2 = |a\s
    b\s|^2 =4$. The fact that $a\s = c\s$ immediately implies
    Inequality~\eqref{i:c3}. Since $C$ and~$D$ are non-overlapping we
    have
    \[
    d_{3}^{\,2} + |{a\s}{d\s}|^2 = d_{3}^{\,2} + |{c\s}{d\s}|^2 =
    |cd|^2 \ge 4,
    \]
    which implies Inequalities~\eqref{i:c1} and~\eqref{i:c2}. The case
    where the three balls are $\{B,C,D\}$ is symmetric.

  \item If the three balls are $\{A,B,D\}$ then, by
    Lemma~\ref{lem:tri-tang}, $a\s = d\s$ and $|a\s b\s|^2 =
    |b\s d\s|^2 =4$. Since $B$, $C$ and~$D$ are non-overlapping, we
    have
    \[
    d_2^{\,2} = |bc|^2 - |{b\s}{c\s}|^2 \ge 4- |{b\s}{c\s}|^2
    \quad \hbox{and} \quad d_3^{\,2} = |cd|^2 - |{c\s}{d\s}|^2 \ge
    4- |{c\s}{d\s}|^2 = 4 - |{c\s}{a\s}|^2.
    \]
    Using $|a\s d\s|^{2} = 0$ and $|a\s b\s|^{2} = 4$, we can bound
    \begin{align*}
      \Delta + d_2^{\,2}+d_3^{\,2} +|a\s d\s|^{2}-|a\s b\s|^2 & \geq
      \Delta + 4-|{b\s}{c\s}|^2 + 4- |{c\s}{a\s}|^{2} - 4 =
      \Delta + 4-|{b\s}{c\s}|^2 - |{c\s}{a\s}|^{2} \\
      \Delta + d_1^{\,2}+d_3^{\,2} +|a\s d\s|^{2}-|b\s c\s|^2 & \geq
      \Delta + d_1^{\,2}+4- |{c\s}{a\s}|^2 -|b\s c\s|^2 \\
      \Delta + d_1^{\,2}+d_2^{\,2} +|a\s d\s|^{2}-|c\s d\s|^2 & \geq
      \Delta + d_1^{\,2}+ 4- |{b\s}{c\s}|^2 -|c\s a\s|^2
    \end{align*}
    Since $c\s$ lies in the disk of diameter~$a\s b\s$, the
    triangle $a\s b\s c\s$ is right or obtuse, and
    \[
    |a\s c\s|^2+|c\s b\s|^2 \le |a\s b\s|^2 = 4.
    \]
    This implies Inequalities~\eqref{i:c1}--\eqref{i:c3}. The case
    where the three balls are $\{A,C,D\}$ is symmetric.
  \end{itemize}

  \medskip

  It remains to handle the case where no three balls in $\{A,B,C,D\}$
  pin $\ell$.  This implies that $\ell$ is tangent to all four balls,
  and the points $a\s, b\s, c\s, d\s$ lie on the unit circle
  around $\ell\s=\ell \cap \ello$ in~$\ello$.  We let
  $\alpha$ be the angle made by $a\s$ and $b\s$ at $\ell \s$,
  $\beta$ the angle made by $b\s$ and $c\s$ at $\ell \s$, and
  $\gamma$ the angle made by $c\s$ and $d\s$ at $\ell \s$, see
  Fig.~\ref{fig:LemmaDistanceABCD}.  All the angles are measured
  counterclockwise (even if they are larger than $\pi$) so that the
  angle made by $a\s$ and $d\s$ is the same as $\alpha+\beta+\gamma$
  modulo $2\pi$.

  We define a function~$g$ by $g(\phi) = \sqrt{2 + 2\cos\phi}$ for
  $\phi \in \R$. We claim that for any three angles $x, y, z$ we have
  \begin{equation}
    \label{i:g}
    g(x + y + z) \le g(x) + g(y) + g(z),
  \end{equation}
  and that the inequality is strict unless two of $x,y,z$ are equal to
  $\pi$ modulo $2\pi$. Indeed, let $f(\phi)=\sqrt{2 - 2\cos\phi} =
  g(\pi-\phi)$ and observe that $f(\phi)$ is the distance between two
  points on the unit circle that make an angle of $\phi$ at the center
  of the unit circle.  The triangle inequality immediately implies
  that for any two angles $\phi$ and $\theta$ we have $f(\phi+\theta)
  \leq f(\phi) + f(\theta)$, where equality holds only if $\phi$ or
  $\theta$ is equal to~$0$ modulo~$2\pi$. Thus, for any three angles
  $x, y, z$ we have
  \begin{align*}
  f\big(\pi - (x + y + z)\big) & = f\big(3\pi - (x + y + z)\big) \\
  & = f\big((\pi - x) + (\pi - y) + (\pi - z)\big)
  \leq f(\pi - x) + f(\pi - y) + f(\pi -  z),
  \end{align*}
  and Inequality~\eqref{i:g} follows, with equality only if two of
  $x,y,z$ are equal to $\pi$ modulo~$2\pi$.

  Consider the angles~$\alpha$, $\beta$, and~$\gamma$.  If $\alpha =
  \beta = \pi$, then by Lemma~\ref{lem:tri-tang}, $\{A, B, C\}$
  already pin~$\ell$, a contradiction.  If $\beta = \gamma = \pi$,
  then $\{B, C, D\}$ already pin~$\ell$, again a contradiction.
  We thus have
  \begin{align}
    \label{eq:g}
    g(\alpha + \beta + \gamma) & \leq g(\alpha) + g(\beta) + g(\gamma),
  \end{align}
  and the inequality is strict unless $\alpha = \gamma = \pi$. In this
  case\footnote{Theorem~\ref{the:min4pinning} actually implies that in
    this case $\ell$ is not pinned at all, the argument here keeps the
    proof self-contained.} $|a\s b\s| = |c\s d\s| =
  2$, and therefore $|ab| > 2$ and~$|cd| > 2$.

  We observe that
  \begin{align*}
    |ab|^2 & = d_1^{\,2} + 2 - 2\cos\alpha = d_{1}^{2} + 4 -
    g(\alpha)^{2},\\
    |bc|^2 & = d_2^{\,2} + 2 - 2\cos\beta = d_{2}^{2} + 4 -
    g(\beta)^{2},\\
    |cd|^2 & = d_3^{\,2} + 2 - 2\cos\gamma = d_{3}^{2} + 4 -
    g(\gamma)^{2}.
  \end{align*}
  Since the balls are non-overlapping, this implies $d_{1} \geq
  g(\alpha)$, $d_{2} \geq g(\beta)$, and $d_{3} \geq g(\gamma)$.  In
  particular, when Inequality~(\ref{eq:g}) is not strict, then $d_{1}
  > g(\alpha)$ and $d_{3} > g(\gamma)$.

  Recall that $\Delta = 2(d_{1}d_{2} + d_{1}d_{3} + d_{2}d_{3})$ and
  set $\Delta' = 2\big(g(\alpha)g(\beta) + g(\alpha)g(\gamma) +
  g(\beta)g(\gamma)\big) \leq \Delta$.
  We can write
  \begin{align*}
    |ad|^2 & = (d_{1} + d_{2} + d_{3})^{2} + 4 - g(\alpha + \beta + \gamma)^{2} \\
    & = d_{1}^{\,2} + d_{2}^{\,2} + d_{3}^{\,2} +
    \Delta + 4 - g(\alpha+\beta+\gamma)^{2} \\
    & \geq d_{1}^{\,2} + d_{2}^{\,2} + d_{3}^{\,2} +
    \Delta' + 4 - g(\alpha+\beta+\gamma)^{2}.
  \end{align*}
  We can now complete the proof.
  \begin{align}
    |ad|^2 & \geq d_{1}^{\,2} + d_{2}^{\,2} + d_{3}^{\,2} +
    \Delta' + 4 - g(\alpha+\beta+\gamma)^{2} \nonumber\\
    & \geq d_{1}^{\,2} + g(\beta)^{2} + g(\gamma)^{2} +
    \Delta' + 4 - g(\alpha+\beta+\gamma)^{2}
    \label{eq:g-2} \\
    & = \big(g(\alpha) + g(\beta) + g(\gamma)\big)^{2}
    - g(\alpha+\beta+\gamma)^{2}  + d_{1}^{\,2} + 4 - g(\alpha)^{2} \nonumber\\
    & \geq d_{1}^{\,2} + 4 - g(\alpha)^{2} = |ab|^2
    \label{eq:g-3}
  \end{align}
  If Inequality~(\ref{eq:g-3}) is not strict, then
  Inequality~(\ref{eq:g-2}) is strict, and so $|ad| > |ab|$.
  \begin{align}
    |ad|^2 & \geq d_{1}^{\,2} + d_{2}^{\,2} + d_{3}^{\,2} +
    \Delta' + 4 - g(\alpha+\beta+\gamma)^{2} \nonumber\\
    & \geq g(\alpha)^{2} + d_{2}^{\,2} + g(\gamma)^{2} +
    \Delta' + 4 - g(\alpha+\beta+\gamma)^{2} \label{eq:g-4} \\
    & = \big(g(\alpha) + g(\beta) + g(\gamma)\big)^{2} -
    g(\alpha+\beta+\gamma)^{2}  + d_{2}^{\,2} + 4 - g(\beta)^{2}
    \nonumber \\
    & \geq d_{2}^{\,2} + 4 - g(\beta)^{2} = |bc|^2
    \label{eq:g-5}
  \end{align}
  Again, if Inequality~(\ref{eq:g-5}) is not strict, then
  Inequality~(\ref{eq:g-4}) is, and we have $|ad| > |bc|$.
  By symmetry, we also obtain $|ad|>|cd|$.
\end{proof}

\section{Conjectures on four unit balls and two lines}
\label{sec:conjecture}

We conjecture that the geometric permutations $\p ABCD.$ and $\p
ACDB.$ are incompatible for non-overlapping unit balls:
\begin{conjecture}
  \label{conj:abcd-acdb}
  There is no set of four non-overlapping unit balls in~$\R^3$
  admitting the geometric permutations $\p ABCD.$ and $\p ACDB.$.
\end{conjecture}

\noindent
As we have seen in the proof of Theorem~\ref{the:main} in
Section~\ref{sec:at-most-three}, Conjecture~\ref{conj:abcd-acdb} would
imply that a family of at least four non-overlapping unit balls
in~$\R^{d}$ has at most two geometric permutations, settling our
question entirely. In this section, we study what a counter-example to
our conjecture would look like.

\bigskip

We first employ the shrinking technique to obtain a configuration
where both line transversals are pinned.

\begin{lemma}
  \label{lem:shrink-conj}
  If Conjecture~\ref{conj:abcd-acdb} is false then there exist four
  non-overlapping unit balls in~$\R^3$ that pin two lines realizing
  the geometric permutations $\p ABCD.$ and $\p ACDB.$.
\end{lemma}
\begin{proof}
  Consider four non-overlapping unit balls~$\F = \{A, B, C, D\}$
  in~$\R^{3}$ that admit line transversals with orders~$\p ABCD.$
  and~$\p ACDB.$.  We uniformly shrink the four balls about their
  centers.  By Lemma~\ref{lem:shrinking}, we will reach a radius~$t_1
  > 0$ where the transversal~$\sigma_1$ for one of the two orders is
  pinned, while a transversal for the other order still exists.  For
  each ball $X \in \F$, we pick a point $x_0 \in X \cap \sigma_1$.  We
  continue shrinking the balls, but now we shrink $X$ with homothety
  center~$x_{0}$.  By Lemma~\ref{lem:shrinking}, we will reach a
  radius~$t_2 > 0$ where the transversal~$\sigma_2$ for the second
  order is pinned.  Since the points~$x_0$ lie in the shrunken balls,
  $\sigma_1$ is still a transversal, and since the balls of
  radius~$t_2$ lie inside the balls of radius~$t_1$, $\sigma_1$ is
  still pinned.  By scaling the balls back to unit radius, we obtain
  the configuration announced by the lemma.
\end{proof}

\subsection{Configurations with two pinned transversals}

In this section we restrict the geometry of configurations as in
Lemma~\ref{lem:shrink-conj}. We start with some geometric
preliminaries. Throughout this section we will be dealing with
families of at most four balls.

\begin{lemma}
  \label{lem:xyz-angle}
  If three non-overlapping unit balls $X, Y$ and $Z$ in $\R^{3}$ admit
  a line transversal with order~$\p XYZ.$, then the angles
  $\angle(yxz)$ and $\angle(xzy)$ are acute.
\end{lemma}
\begin{proof}
  Using Lemma~\ref{lem:shrinking}, we shrink the balls uniformly around
  their centers until a line transversal $\ell$ with order~$\p XYZ.$
  is pinned. By Lemma~\ref{lem:tri-tang}, $\ell$ is parallel to~$xz$
  and $y$ lies inbetween $x$ and $y$ in the projection on~$\ell$,
  implying the claim.
\end{proof}

\begin{lemma}
  \label{lem:xyz,xzy}
  If three non-overlapping unit balls $X, Y$ and $Z$ in $\R^{3}$ admit
  two line transversals with the orders $\p XYZ.$ and $\p XZY.$, then
  the triangle $\triangle xyz$ is acute and $|yz|<2\sqrt{2}$.
\end{lemma}
\begin{proof}
  Lemma~\ref{lem:xyz-angle} implies that $\triangle xyz$ is acute. To
  prove the last statement, we shrink the balls uniformly around their
  centers.  By Lemma~\ref{lem:shrinking}, there exists $t > 0$ such that
  without loss of generality the following holds: $X$, $Y$, $Z$ are
  now balls with radius~$t \leq 1$ pinning a line transversal~$\ell$
  with order~$\p XYZ.$, and the balls have a second line
  transversal~$\ell'$ with order~$\p XZY.$.  By
  Lemma~\ref{lem:tri-tang}, $\ell$ lies in the plane containing~$x, y,
  z$, and separates $y$ from $x$ and~$z$. Let $x', y', z'$ be the
  projection of $x, y, z$ onto~$\ell$, they appear in this order
  along~$\ell$. If $|yz| \geq 2\sqrt{2}$, then $|y'z'| \geq 2$, and
  there is a plane orthogonal to~$\ell$ that separates $Z$ on one side
  from $X$ and~$Y$ on the other side. But that contradicts the
  existence of~$\ell'$.
\end{proof}

We can now state our restriction on a possible counter-example of
Conjecture~\ref{conj:abcd-acdb}.
\begin{theorem}
  \label{the:double-pinning}
  If four non-overlapping unit balls $\{A, B, C, D\}$ in $\R^3$ pin
  two lines with the geometric permutations $\p ABCD.$ and~$\p ACDB.$,
  then these lines are not pinned by a proper subset of the balls.
\end{theorem}
\begin{proof}
  Consider four non-overlapping unit balls $\{A,B,C,D\}$ in $\R^3$
  that pin two lines $\sigma_1$ and $\sigma_2$ realizing,
  respectively, the geometric permutations $\p ABCD.$ and~$\p ACDB.$.
  Let us first remark that
  \begin{equation}\label{eq:abs1}
    \pi/4 < \angle(\vc{ab}, \vc{\sigma_1}) < \pi/2
  \end{equation}
  Indeed, since $\sigma_1$ meets $A$ before $B$, we have
  $\angle(\vc{ab}, \vc{\sigma_1}) < \pi/2$. Moreover, if
  $\angle(\vc{ab}, \vc{\sigma_1}) \leq \pi/4$, since
  Lemma~\ref{lem:angle} yields that $\angle(\vc{\sigma_1}, \vc{bd})<
  \pi/4$, we would have $\angle(\vc{ab},\vc{bd}) \leq
  \angle(\vc{ab},\vc{\sigma_1}) + \angle(\vc{\sigma_1}, \vc{bd}) <
  \pi/2$, a contradiction with $\triangle abd$ being acute by
  Lemma~\ref{lem:xyz,xzy}. Let us also remark that
  \begin{equation}\label{eq:ab}
    |ab| < 2\sqrt{2}
  \end{equation}
  as otherwise any line meeting~$A$ before~$B$ would make an angle less
  than~$\pi/4$ with~$\vc{ab}$, contradicting
  Equation~\eqref{eq:abs1}.

  \begin{figure}[ht]
    \centerline{\includegraphics{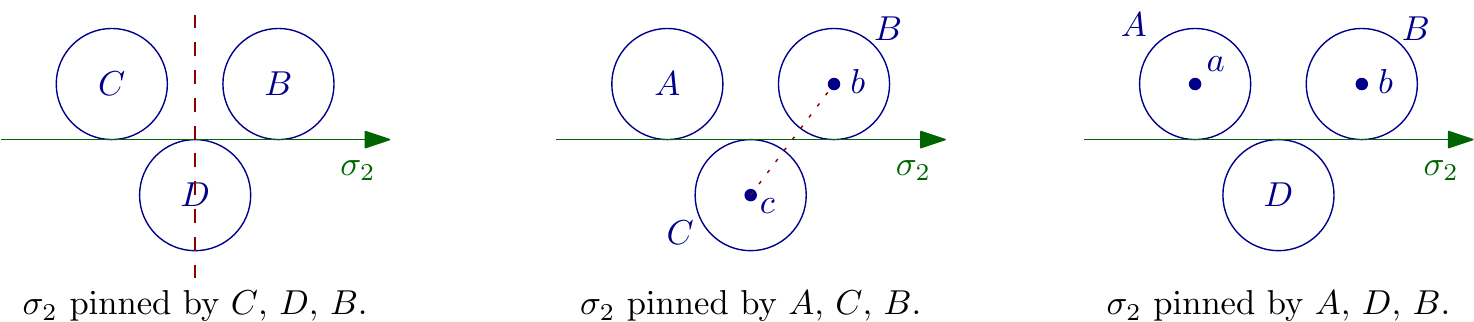}}
    \caption{Proof of Theorem~\ref{the:double-pinning}}
    \label{fig:double-pinning}
  \end{figure}

  \medskip

  Let us assume for a contradiction that~$\sigma_2$ is pinned by three
  of the four balls. We first remark that~$\sigma_2$ can only be
  pinned by $A$, $D$, $B$ (refer to Fig.~\ref{fig:double-pinning}):
  \begin{denseitems}
  \item If $\sigma_2$ is pinned by $C$, $D$, and~$B$, then by
    Lemma~\ref{lem:tri-tang}, $\sigma_2$ lies in the plane of $bcd$
    and is parallel to~$cb$.  Since $\sigma_2$ meets~$A$ before~$C$,
    the bisecting plane of $C$ and~$B$ separates $B$ from $A$ and $C$,
    a contradiction to the existence of~$\sigma_1$.

  \item If $\sigma_2$ is pinned by $A$, $C$, $D$, then it is the only
    line meeting the three balls in this order.  But then $\sigma_2 =
    \sigma_1$, a contradiction.

  \item If $\sigma_2$ is pinned by $A$, $C$, $B$, then by
    Lemma~\ref{lem:tri-tang} it lies in the plane of $acb$ and is
    parallel to~$ab$. Lemma~\ref{lem:xyz,xzy} implies that $|bc| <
    2\sqrt{2}$ and so $\angle(\vc{cb}, \vc{\sigma_2}) > \pi/4$.
    Lemma~\ref{lem:angle} yields, however, that $\angle(\vc{cb},
    \vc{\sigma_2}) < \pi/4$, a contradiction.
  \end{denseitems}
  So assume that $\sigma_2$ is pinned by $A$, $D$, and $B$. By
  Lemma~\ref{lem:tri-tang}, $\sigma_2$ lies in the plane of $adb$ and
  is parallel to~$\vc{ab}$. Thus, $\angle{(\vc{ab},\vc{cb})} =
  \angle{(\vc{\sigma_2},\vc{cb})} < \pi/4$ by Lemma~\ref{lem:angle}.
  It follows that $c$ is contained in $\cone(b,\vc{ba},\pi/4)$, where
  $\cone(u, \vec{v}, \alpha)$ denotes the cone of all points~$p$ such
  that $\angle(\vc{up}, \vec{v}) \leq \alpha$.  Since $\sigma_1$
  intersects $C$ after $A$ and~$B$, the center~$c$ must lie in the
  cone with apex~$a$ spanned by the ball of radius~$2$ and center~$b$.
  With $\gamma = \arcsin(2/ab)$, this cone is $\cone(a, \vc{ab},
  \gamma)$. By Equation~\eqref{eq:ab}, $2 \leq |ab|<2\sqrt{2}$, and so
  $\pi/4 < \gamma \leq \pi/2$.  Let $p$ be a point in the plane
  containing $a$, $b$ and $c$ and such that $\angle(bap) = \gamma$ and
  $\angle(abp) = \pi/4$ (see Fig.~\ref{fig:ADB}). Notice that $c$ lies
  in the triangle~$\triangle abp$. We claim that any point in this
  triangle is at distance less than $2$ from $a$ or $b$. In particular,
  there is no way to place $c$ so as to make $A$, $B$, $C$
  non-overlapping unit balls and $\sigma_2$ cannot be pinned by $A$,
  $D$, and $B$, or, more generally, by three of the balls.

  \begin{figure}[ht]
    \centerline{\includegraphics{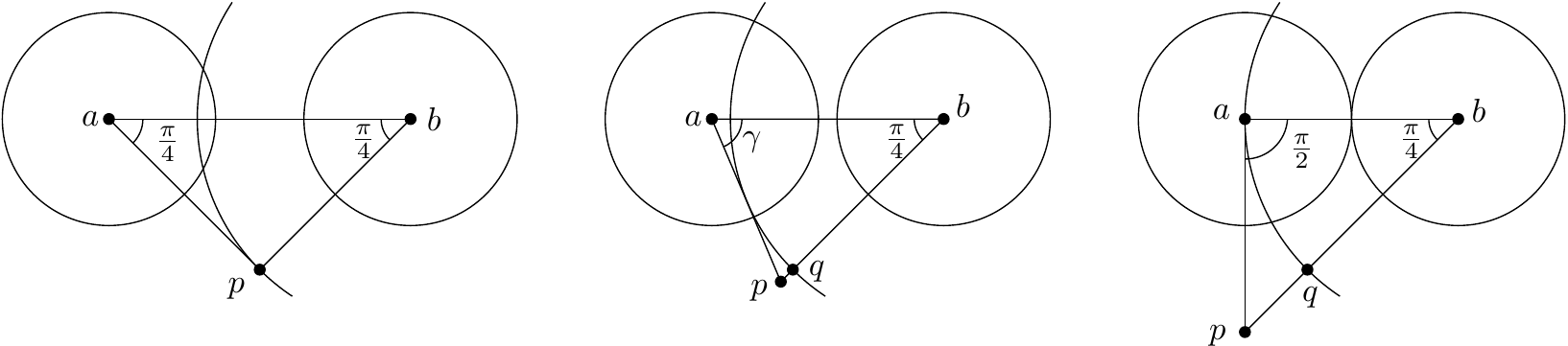}}
    \caption{$\triangle abp$ for three values of~$\gamma$.}
    \label{fig:ADB}
  \end{figure}

  It remains to prove the claim on~$\triangle abp$. Since $\gamma >
  \pi/4$ we have $|bp| > 2$.  Let $q$ be the point on the segment $bp$
  with~$|bq| = 2$.  Since the segment~$ap$ touches the circle of
  radius~$2$ around~$b$, we have $\angle(aqp) > \pi/2$, and so $|aq| <
  |ap|$. We claim that $|ap| < 2$. Indeed, the law of sines gives
  $|ab|/\sin(3\pi/4-\gamma) = |ap|/\sin(\pi/4)$, and so
  \[
  |ap| = \frac{\sin(\pi/4)}{\sin(\pi/4 + \gamma)} |ab|
  = \frac{\frac 12 \sqrt{2}}{\sin{(\pi/4+\gamma)}}
  \frac{2}{\sin{\gamma}}
  = \frac{\sqrt{2}}{\sin(\pi/4+\gamma)\sin \gamma}
  = \frac{2}{\sin \gamma(\sin \gamma + \cos \gamma)}
  \]
  Define $f(x) = \sin x(\sin x + \cos x) = \sin^{2}x + \frac 12 \sin
  2x$.  Since $f'(x) = \sin 2x + \cos 2x$, the function~$f$ is
  (strictly) increasing from $x = 0$ to $x = 3\pi/8$, and (strictly)
  decreasing from $x = 3\pi/8$ to $x = 7\pi/8$. Since $f(\pi/4) =
  f(\pi/2) = 1$, it follows that $f(x) > 1$ for $\pi/4 < x < \pi/2$;
  this proves our claim that $|ap| < 2$. Now, if a point~$u \in \triangle
  apb$ lies to the left of the vertical line through~$q$, then it is
  at distance less than~$2$ from~$a$, if $u$ lies to the right of~$q$
  then it has distance less than~$2$ from~$b$.

  \medskip

  It follows that $\sigma_2$ cannot be pinned by any three of the
  balls, and is tangent to all four. We now assume, for a
  contradiction, that $\sigma_1$ is pinned by three of the balls.
  Again, we easily dismiss three of the cases:
  \begin{denseitems}
  \item If $\sigma_1$ is pinned by $B$, $C$, and~$D$, then by
    Lemma~\ref{lem:tri-tang} $\sigma_1$ lies in the plane of $bcd$ and
    is parallel to~$bd$.  Since $\sigma_1$ meets~$A$ before~$B$, the
    bisecting plane of $B$ and~$D$ separates $D$ from $A$ and $B$, a
    contradiction to the existence of~$\sigma_2$.

  \item If $\sigma_1$ is pinned by $A$, $C$, $D$, then $\sigma_2 =
    \sigma_1$, a contradiction.

  \item If $\sigma_1$ is pinned by $A$, $B$, $D$, then it lies in the
    plane of $abd$ and is parallel to~$ad$.  Since $|bd| < 2\sqrt{2}$
    by Lemma~\ref{lem:xyz,xzy}, we have $\angle(\vc{bd},
    \vc{\sigma_1}) > \pi/4$, contradicting Lemma~\ref{lem:angle}.
  \end{denseitems}
  \begin{figure}[ht]
    \centerline{\includegraphics{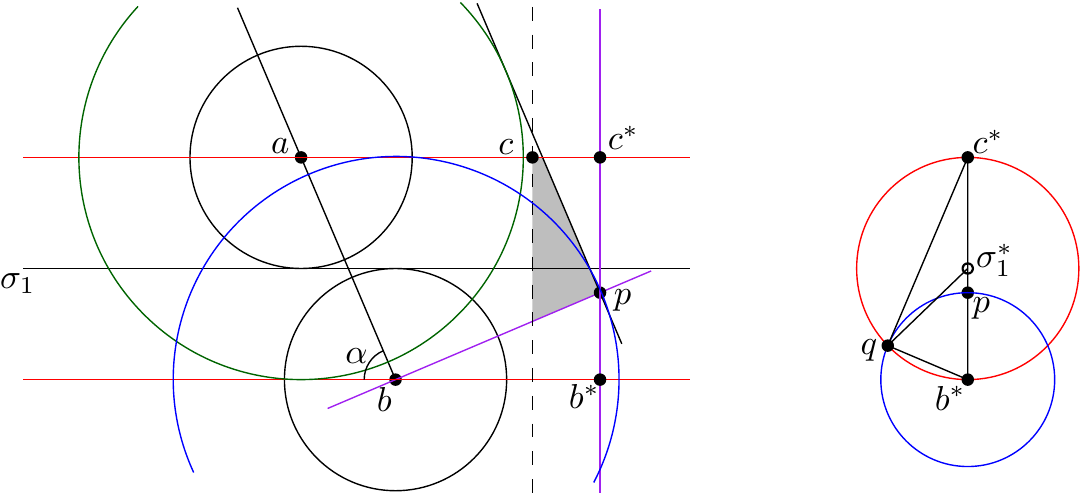}}
    \caption{When $\sigma_1$ is pinned by $A$, $B$, $C$}
    \label{fig:ABC}
  \end{figure}
  So $\sigma_1$ must be pinned by $A$, $B$, and~$C$ and lie in the
  plane spanned by~$abc$, see Fig.~\ref{fig:ABC}. Since $\sigma_1$ is
  a transversal, the center~$d$ of ball~$D$ must lie inside the
  cylinder~$\C$ of radius~one with axis~$\sigma_1$.  Since $\sigma_2$
  meets $A$, $D$, and $B$ in this order, $d$ must also lie in the
  cylinder of radius $2$ with axis~$ab$. Since $\angle(abd) < \pi/2$
  by Lemma~\ref{lem:xyz,xzy}, $d$ lies above the plane orthogonal
  to~$ab$ through~$b$, and since $\sigma_1$ meets $D$ after $C$, $d$
  lies to the right of the plane orthogonal to $\sigma_1$ through~$c$.
  In the projection on the $abc$-plane, this restricts $d$ to the
  shaded area in Fig.~\ref{fig:ABC}.  Let~$p$ be the rightmost point
  of this feasible region for~$d$, that is, the point in the
  $abc$-plane such that $|bp| = 2$ and $\angle(abp) = \pi/2$. For any
  point $u$, let $\sigmao_1(u)$ be the plane orthogonal to $\sigma_1$
  passing through~$u$. The center~$d$ lies in the cylinder~$\C$
  between the planes~$\sigmao_1(c)$ and~$\sigmao_1(p)$.

  We will now show that any point~$d$ in this cylinder has distance
  less than two from $b$ or~$c$, and so $D$ cannot be non-overlapping
  with $B$ and~$C$, a contradiction.  Clearly it suffices to show this
  for the disk~$S$ of radius one around~$\sigma_1$ in the
  plane~$\sigmao_1(p)$.  Let $q$ be a point on the boundary of~$S$
  with $|bq|=2$ (see right side of Fig.~\ref{fig:ABC}). It suffices to
  show that $|cq| < 2$. Let $\alpha = \angle(\vc{ab},
  \vc{\sigma_1})$.  By Equation~\eqref{eq:abs1},
  $\pi/4<\alpha<\pi/2$.  Since
  $|qb\s|=|pb\s|=2\sin{(\pi/2-\alpha)}=2\cos{\alpha}$ and $\angle(b\s
  \!qc\s)=\pi/2$, $|qc\s|^2=4-4\cos^2{\alpha}= 4\sin^2{\alpha}$.  We
  have $|cc\s| =|ac\s| -|ac|= |ab|\cos{\alpha}+|bb^{\s}|-|ac|$, and
  with $|ab|=2/\sin{\alpha}$, $|bb\s| = 2\cos(\pi/2-\alpha) =
  2\sin\alpha$, and $|ac|\geq 2$ this implies $|cc\s| \leq
  2\cot{\alpha} + 2\sin{\alpha} - 2$.  Thus $|cq|^2 = |qc\s|^2 +
  |cc\s|^2 \leq 4 \sin^{2}{\alpha} + (2\cot{\alpha} + 2\sin{\alpha}
  -2)^2$.

  We have
  \begin{align}
    \sin^{2}{x} + \big(\cot {x} + \sin{x} - 1 \big)^2 - 1
    & =\big(\cot{x}+\sin{x}-1\big)^2-\cos^{2} x \nonumber\\
    & = \big(\cot{x}+\sin{x}+\cos{x}-1 \big)
    \big(\cot{x}+\sin{x}-\cos{x}-1 \big) \nonumber \\
    & = \big(\cot{x}+\sin{x}+\cos{x}-1 \big)
    (\cot x - 1)(1 - \sin x).
    \label{eq:abc1}
  \end{align}
  On the interval $\pi/4 <x< \pi/2$, we have $\sqrt{2}/2 < \sin x <
  1$, $\sqrt{2}/2 > \cos x > 0$, and $1 > \cot x > 0$, implying $\cot
  x -1 < 0$ and $1 - \sin x > 0$.  Furthermore
  \[
  \sin{x}+\cos{x}
  =\sqrt{2}\sin{(\frac{\pi}{4}+x)} > 1 \qquad \text{for}\; \pi/4 < x < \pi/2,
  \]
  and so the first term in Equation~(\ref{eq:abc1}) is positive. It
  follows that $\sin^{2} x + (\cot x + \sin x - 1)^{2} < 1$.  This
  implies $|cq|^2 < 4$, and we arrived at the contradiction for this
  final case.
\end{proof}

\subsection{Minimal pinnings by four balls}

By Lemma~\ref{lem:shrink-conj} and Theorem~\ref{the:double-pinning},
if Conjecture~\ref{conj:abcd-acdb} is false then there exist four
non-overlapping unit balls with two transversals with the orders~$\p
ABCD.$ and~$\p ACDB.$ that are both pinned by the four balls but no
three of them. We now analyze the geometry of such \emph{minimal
  pinnings} by four balls (Theorem~\ref{the:min4pinning}) and derive a
statement equivalent to Conjecture~\ref{conj:abcd-acdb} but more
restrictive (Conjecture~\ref{conj:double-pinning}).
\begin{figure}[ht]
  \centerline{\includegraphics{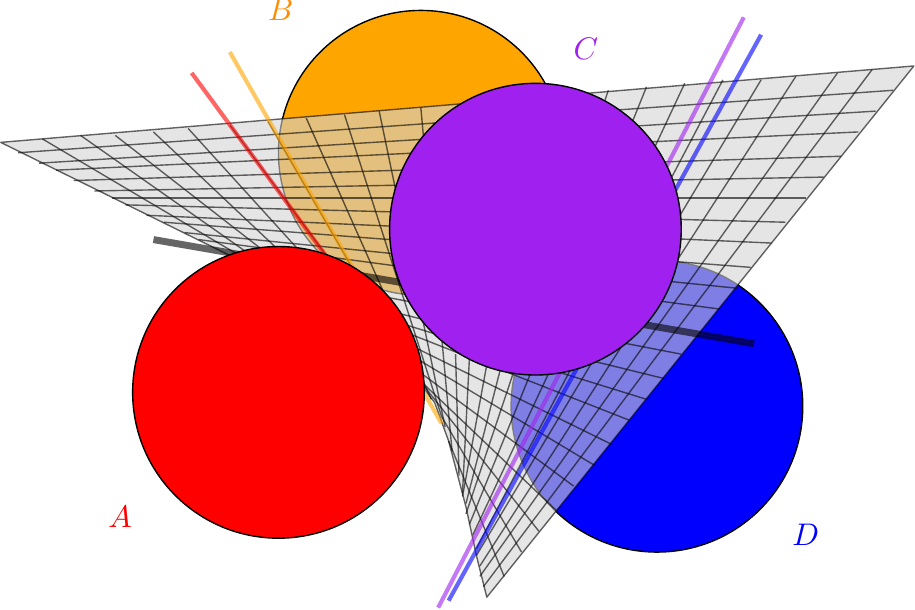}}
  \caption{An alternating hyperboloidal configuration\label{fig:althyp}}
\end{figure}

If $X$ is a ball tangent to a line $\ell$, the
\emph{ridge}~$r_\ell(X)$ of $X$ with respect to $\ell$ is the line
tangent to $X$ and perpendicular to $\ell$ in $X \cap \ell$. We say
that four or more lines are in \emph{hyperboloidal configuration} if
they are all contained in the same family of rulings of a hyperbolic
paraboloid or a hyperboloid of one sheet (see~\cite{hcv} for a
classical discussion of such configurations). An \emph{alternating
  hyperboloidal configuration} is a pair $(\F,\ell)$ where $\F$ is a
family of four unit balls balls and $\ell$ is a line tangent to every
member of $\F$ satisfying the two following conditions:
\begin{denseitems}
\item[(i)] the ridges $\{r_\ell(X) \mid X \in \F\}$ are in
  hyperboloidal configuration, witnessed by a hyperbolic
  paraboloid~$\mathcal{H}$,
\item[(ii)] the normals to $\mathcal{H}$ at its tangency point with
  the balls of $\F$, directed towards the center of that ball and
  ordered along $\ell$, point to alternating sides of $\mathcal{H}$.
\end{denseitems}
Since in an alternating hyperboloidal configuration the ridges are all
perpendicular to $\ell$, and therefore parallel to a common plane, the
quadric they span can only be a hyperbolic paraboloid. Condition~(i)
then forces the line~$\ell$ to intersect~$\mathcal{H}$ in at least
four points and therefore $\ell \subset \mathcal{H}$. This in turn
implies that every ball $X \in \F$ is tangent to $\mathcal{H}$ in $X
\cap \ell$ as the tangent plane to both $X$ and $\mathcal{H}$ in $X
\cap \ell$ contains the two lines (but $X$ does not need to intersect
$\mathcal{H}$ in a single point).  That and the fact that a hyperbolic
paraboloid~ separates~$\R^3$ into two connected components makes
condition~(ii) well-defined.

\begin{theorem}
  \label{the:min4pinning}
  If a family $\F$ of $4$ non-overlapping unit balls in $\R^3$
  minimally pins a line $\ell$ then $(\F,\ell)$ is an alternating
  hyperboloidal configuration.
\end{theorem}
\begin{proof}
  Let $\sigma$ be a line minimally pinned by $\p ABCD.$ in that order.
  We assume that $\sigma$ coincides with the $z$-axis. The center of
  ball $X$ is denoted by $x$ and its contact point with $\sigma$ is
  denoted by $x'$. We parameterize the space of lines by $\R^4$ using
  the coordinates of the intersections with the planes $z=0$ and
  $z=1$: the parameters $(u_1,u_2,u_3,u_4)$ corresponds to the line
  through $(u_1,u_2,0)$ and $(u_3,u_4,1)$. The lines lying in a plane
  with constant $z$ are not represented, but this will not be an
  issue. The point $(0,0,0,0)$ corresponds to~$\sigma$.

  To every ball $X \in \{A,B,C,D\}$ we associate the \emph{screen}
  $S_\sigma(X)$ that is the intersection of the closed halfspace
  bounded by the tangent plane to $X$ in $x'$ that contains $X$, and
  the plane perpendicular to $\sigma$ in $x'$. The screen
  $S_\sigma(X)$ is a halfplane that lies in a plane perpendicular to
  $\sigma$ and is bounded, in that plane, by the ridge~$r_\sigma(X)$.
  Now, the line transversals to $S_\sigma(X)$ form, in our $\R^4$, a
  halfspace $H(X)$ bounded by a hyperplane through the
  origin~\cite[pp.~4--5]{aronov2011lines}. We let $n(X)$ denote the outer
  normal of $H(X)$ and observe that the boundary of $H(X)$ is the set
  of lines intersecting the ridge $r_\sigma(X)$.

  Now let $I=H(A) \cap H(B) \cap H(C) \cap H(D)$. A necessary
  condition for the balls $\{A,B,C,D\}$ to pin $\sigma$ is that $I$
  has empty interior~\cite[Lemma~9]{cheong2012lower}. This implies
  that the family of normals $\{n(A), n(B), n(C), n(D)\}$ is linearly
  dependant (since, in $\R^4$, four halfspaces with linearly
  independent normals intersect with non-empty interior). It could be
  that two, three or all four vectors are \emph{minimally} linearly
  dependent. Geometric interpretations of these situations were given
  in~\cite[Lemma~15]{aronov2011lines}:

  \begin{denseitems}
  \item If two normals are dependent then the two corresponding ridges
    are equal. This cannot happen for non-overlapping balls.
  \item If three normals are dependent, then it must be that the three
    corresponding ridges are either coplanar with or concurrent on
    $\sigma$. Concurrency is again ruled out for non-overlapping
    balls, and we rule out coplanarity in the next paragraph.
  \item If no three normals are dependent then the four ridges
    are in hyperboloidal configuration (the other case with concurrent
    ridges can again not occur in our situation).
  \end{denseitems}

  Let us observe that the case where three normals are linearly
  dependent cannot correspond to a minimal pinning of $\sigma$ by the
  four balls. As mentioned, the three corresponding ridges must lie,
  together with $\sigma$, in some plane $\Pi$. Let us denote them
  $r_1$, $r_2$ and $r_3$ in the order in which $\sigma$ meets them,
  and let $X_i$ be the ball corresponding to~$r_i$. Since a triple of
  balls does not suffice to pin~$\sigma$, $\Pi$ does not separate
  $X_2$ from $X_1$ and~$X_3$; by symmetry we can assume that either
  all three balls are on the same side of $\Pi$ or $\Pi$ separates
  $X_1$ and $X_2$ from $X_3$.  Since $I$ has empty interior, the
  fourth ridge also lies in the plane~$\Pi$. Then, either three of the
  balls pin~$\sigma$ or all four do not, a contradiction.

  We must therefore be in the situation where the four ridges are in
  hyperboloidal configuration. Observe that since any three normals
  are linearly independent, the intersection $I$ must be exactly a
  line in~$\R^{4}$.  That line $I \subset \R^4$ corresponds to the set
  of lines intersecting all four ridges. Let $Q$ denote the quadric
  formed by the union (in $\R^3$) of these line transversals to the
  four ridges. Let us orient $Q$, that is choose an outward normal
  (defined continuously over all of $Q$).  There are two connected
  components in $\R^3\setminus Q$, which we call \emph{sides}
  of~$Q$. As we move a point $p$ along~$\sigma$, the outward normal
  of~$Q$ in~$p$ keeps pointing into the same side but rotates
  continuously around $\sigma$; as $p$ ranges over all of $\sigma$,
  that normal turns by a total angle of~$\pi$.  For each ball $X \in
  \{A,B,C,D\}$, $S(X)\setminus r(X)$ is contained either in the
  positive side or in the negative side of~$Q$.

  Now consider the orthogonal projection of the four screens on the
  plane $z=0$. The circular order in which the projections of the
  ridges appear matches the order in which $\sigma$ meets the screens;
  indeed, the projection of ridge $r(X)$ is simply the trace in $z=0$
  of the plane tangent to~$Q$ in $x'$, and we observed that the
  tangent plane turns continuously, and by a total angle of~$\pi$, as
  the contact point ranges over all of~$\sigma$.  Moreover, (the
  relative interiors of) any two consecutive screens are contained in
  opposite sides of~$Q$ as otherwise we can perturb the plane $z=0$ so
  that the projections of the four screens intersect with non-empty
  interior, and $I$ cannot have empty interior. Altogether, this
  proves that a quadruple of balls minimally pinning a line must
  form an alternating hyperboloidal configuration.
\end{proof}

Theorem~\ref{the:min4pinning}, Lemma~\ref{lem:shrink-conj}, and
Theorem~\ref{the:double-pinning} imply that
Conjecture~\ref{conj:abcd-acdb} can be reformulated in the following
form:
\begin{conjecture}
  \label{conj:double-pinning}
  There is no set $\F$ of four non-overlapping unit balls in~$\R^3$
  with two line transversals $\sigma_1$, $\sigma_2$ that realize the
  geometric permutations $\p ABCD.$ and $\p ACDB.$ and such that
  $(\F,\sigma_1)$ and $(\F,\sigma_2)$ are alternating hyperboloidal
  configurations.
\end{conjecture}

\section{Concluding remarks}

Conjectures~\ref{conj:abcd-acdb} and~\ref{conj:double-pinning} can be
expressed by asking whether a system of low-degree polynomial
equations and inequalities in a small number of variables has a
solution (see Appendix~\ref{sec:semialgebraic}).  In principle, such
systems can be solved by computer algebra software based on Gr\"obner
basis computation such as the \textsc{raglib} maple library.
Inequalities are harder to handle than equalities by these solvers, so
this is one reason why Theorem~\ref{the:min4pinning},
Lemma~\ref{lem:shrink-conj}, and Theorem~\ref{the:double-pinning} are
interesting: Conjecture~\ref{conj:double-pinning} replaces most of the
inequalities in Conjecture~\ref{conj:abcd-acdb} by equalities.

Our attempts using algebraic solver software were inconclusive.  These
questions may constitute interesting challenges for the computer
algebra community.

\medskip

We speculate that Theorem~\ref{the:min4pinning} can be turned into an
equivalence, at least when the balls are disjoint. One approach could
be to remark that if $(\F,\sigma)$ is an alternating hyperboloidal
configuration where $\sigma$ realizes the geometric permutation $\p
ABCD.$, then the direction $\vec{u}$ of $\sigma$ is on the boundary of
the so-called \emph{cone of directions} of transversals to the triples
$\p ABC.$ and $\p BCD.$ (using ideas such as
eg.~\cite[Proposition~3]{borcea2008line}). These cones are strictly
convex in $\vec{u}$, are mutually tangent in $\vec{u}$ (their common
supporting great circle is the set of directions of transversals to
the four ridges) and the alternating property of the configuration
$(\F,\sigma)$ ensures that this tangency is \emph{external}.  Spelling
out this outline requires non-trivial technical developments, all the
more if one cares for the setting of non-overlapping balls, and is not
needed for our main result of Section~\ref{sec:conjecture}; we thus
leave it to the interested reader to check the validity of this
approach.

\section*{Acknowledgments.}

The experiments with the computer algebra software leading to the
polynomial systems given in Appendix~\ref{sec:semialgebraic} were
conducted in collaboration with Guillaume Moroz, to whom the authors
are very grateful.


\newpage
\appendix

\section{Correctness of the pinning method}
\label{sec:holmsen}

We start with a lemma in two dimensions:
\begin{lemma}
  \label{lem:two-not-pinned}
  Let $\F$ be a family of non-overlapping, but not necessarily
  congruent disks in the plane. If $\F$ admits two distinct line
  transversals with the same order, then there is a transversal with
  the same order that intersects the interior of every disk.
\end{lemma}
\begin{proof}
  Consider one transversal~$\ell$ of the two.  The disks project
  along~$\ell$ onto its orthogonal complement~$\ello$ as
  intervals. Since $\ell$ is a transversal, this intersection is not
  empty. If the intersection is an interval, we are done.  If it is a
  point, then $\ell$ is tangent to some of the disks.  The tangent
  disks cannot alternate, as then $\ell$ would be pinned, and no
  second transversal could exist.  It follows that the tangencies on
  the left strictly follow the tangencies on the right, or vice versa,
  and we can slightly rotate~$\ell$ to obtain the desired transversal.
\end{proof}

We now closely follow Holmsen et al.~\cite{holmsen2003helly}.  Let
$H(z)$ be the plane parallel to the $xy$-plane at height~$z$.  For two
non-overlapping unit balls~$A$, $B$ in~$\R^{3}$ and any $z \in \R$,
let $\K(AB,z)$ be the set of angles~$\theta$ such that there is a
directed transversal meeting $A$ before~$B$, lying in $H(z)$, and
making angle~$\theta$ with the positive $x$-axis.
\begin{lemma}
  \label{lem:holmsen}
  Given two non-overlapping unit balls~$A$ and~$B$ in~$\R^3$ with
  $\theta_1 \in \K(AB,z_1)$ and $\theta_2 \in \K(AB, z_2)$.  Then there
  is an $\eps > 0$ such that the interval $[\theta_0 - \eps, \theta_0
    + \eps]\subset \K(AB, z_0)$, where $z_0 = (z_1 + z_2)/2$ and
  $\theta_0$ bisects the smaller angle between $\theta_1$
  and~$\theta_2$.
\end{lemma}
\begin{proof}
  Since the statement of the lemma is invariant under coordinate
  transformations that keep the normal vector of~$H(z)$ fixed, we can
  assume that $A$ has center~$(0,0,0)$ and $B$ has center~$(d, 0, b)$,
  with $d > 0$ and $0 \leq b \leq 2$.  For $b-1 \leq z \leq 1$ the
  intersections $H(z) \cap A$ and $H(z) \cap B$ are two disks of
  radius~$R(z)$ and~$R(z-b)$, respectively, where
  \[
  R(z) = \sqrt{1 - z^{2}}.
  \]
  A directed line in~$H(z)$ meeting~$A$ before~$B$ makes an
  angle~$\alpha \in (-\pi/2, \pi/2)$ with the positive $x$-axis. Such
  transversals exist for $b-1 \leq z \leq 1$, and so we can assume
  $b-1 \leq z_1 \leq z_0 \leq z_2 \leq 1$.  For a fixed~$z$ with $b-1
  \leq z \leq 1$, transversals with orientation~$\alpha$ exist for
  $-G(z) \leq \alpha \leq G(z)$, where
  \[
  G(z) = \arcsin(f(z)) \qquad\text{and}\qquad
  f(z) = \frac{R(z) + R(z-b)}{d},
  \]
  with one exception: If $A$ and $B$ touch, they do so in~$H(b/2)$. In
  that case $f(b/2) = 1$ and $G(b/2) = \pi/2$, but transversals exist
  only for $-\pi/2 < \alpha < \pi/2$.

  To prove the lemma, it suffices to show that $G(z_0) > (G(z_1) +
  G(z_2))/2$.  We will show that in fact the function~$G(z)$ is
  strictly convex by showing that~$G''(z) < 0$.

  The function $G(z)$ is symmetric about $z = b/2$.  When $A$ and
  $B$ touch, then $G(z)$ is not differentiable in $z = b/2$, but
  in that case $G(b/2) = \pi/2$ and this is clearly the maximum.  It
  therefore suffices to show $G''(z) < 0$ for $b/2 < z < 1$. In this
  range, we have $f(z) < 1$, and
  \[
  f'(z) = \frac{R'(z) + R'(z-b)}{d}, \qquad
  f''(z) = \frac{R''(z) + R''(z-b)}{d}, \qquad
  f'''(z) = \frac{R'''(z) + R'''(z-b)}{d},
  \]
  where
  \[
  R'(z) = \frac{-z}{(1-z^{2})^{1/2}}, \qquad
  R''(z) = \frac{-1}{(1-z^{2})^{3/2}}, \qquad
  R'''(z) = \frac{-3z}{(1-z^{2})^{5/2}}, \qquad
  R''''(z) = \frac{-12z^{2}-3}{(1-z^{2})^{7/2}}.
  \]
  We note that $R''(z) < 0$ and $R''''(z) < 0$ for all $-1 < z < 1$,
  which implies that $R'(z)$ and $R'''(z)$ are strictly decreasing in this
  range. Therefore $f'(z)$ and $f'''(z)$ are strictly decreasing in
  the range~$b/2 < z < 1$.  Since $f'(b/2) = f'''(b/2) = 0$, this
  means that $f'(z) < 0$ and $f'''(z) < 0$ for $b/2 < z < 1$.  We have
  next
  \[
  G'(z) = \frac{f'(z)}{(1-(f(z))^{2})^{1/2}} \qquad \text{and} \qquad
  G''(z) = \frac{g(z)}{(1-(f(z))^{2})^{3/2}},
  \]
  where
  \[
  g(z) = f''(z)(1-(f(z))^{2}) + f(z)(f'(z))^{2}.
  \]
  The sign of $G''(z)$ is determined by $g(z)$, which is well defined
  and differentiable at $z = b/2$ even when $A$ and $B$ touch.  Since
  $f'(b/2) = 0$ and $f''(b/2) < 0$ we have $g(b/2) \leq 0$, with
  equality only if $A$ and~$B$ touch. We have
  \[
  g'(z) = f'''(z)(1 - (f(z))^{2}) + (f'(z))^{3} < 0 \qquad
  \text{for $b/2 < z < 1$},
  \]
  since $f'''(z) < 0$ and $f'(z) < 0$. It follows that $g(z)$ is
  strictly decreasing in the range $b/2 < z < 1$, which implies $g(z)
  < g(b/2) \leq 0$.  Consequently $G''(z) < 0$ for $b/2 < z < 1$,
  completing the proof.
\end{proof}

We now have all the necessary tools.
\begin{proof}[Proof of Lemma~\ref{lem:shrinking}.]
  The lemma is true with $t\s = 0$ if the centers are all collinear,
  so assume this is not the case.  We set $t\s > 0$ to be the infimum
  over all radii~$t$ where $\F(t)$ has a transversal with the given
  order.  It follows from the compactness of the balls that $\F(t\s)$
  has a transversal~$\ell_1$ with the correct order.  Assume for a
  contradiction, that $\F(t\s)$ has a second line transversal~$\ell_2
  \neq \ell_1$ with the same order.  We will argue that then there is
  another transversal~$\ell$ with the same order that intersects the
  interior of every ball in~$\F(t\s)$.  This implies that there is an
  $\eps > 0$ such that $\ell$ is a transversal for the family
  $\F(t\s-\eps)$ as well, a contradiction.

  If $\ell_1$ and $\ell_2$ are parallel, then the entire strip bounded
  by the two lines intersects all balls, and we can choose $\ell$ to
  be any line inbetween.  Assume next that $\ell_1$ and $\ell_2$ are
  not parallel, and choose a coordinate system where they are parallel
  to the $xy$-plane.  Let $\ell_i$, for $i \in \{1,\,2\}$, lie in the
  plane~$H(z_i)$ and make angle~$\theta_i$ with the positive~$x$-axis.
  Let $z_0 = (z_1 + z_2)/2$ and let $\theta_0$ be the angle
  bisecting~$\theta_1$ and~$\theta_2$ as in Lemma~\ref{lem:holmsen}.

  For every pair $1 \leq i < j \leq j$, Lemma~\ref{lem:holmsen}
  guarantees the existence of an~$\eps_{ij} > 0$ such that the
  interval $[\theta_0 - \eps_{ij}, \theta_0 + \eps_{ij}] \subset
  \K(B_{i}(t\s)B_{j}(t\s), z_{0})$.  Setting $\eps = \min_{i < j}
  \eps_{ij}$, we have
  \[
    [\theta_0 - \eps, \theta_0 + \eps] \subset \bigcap_{i < j}
    \K(B_{i}(t\s)B_{j}(t\s), z_{0}).
  \]
  Consider now the family of disks in~$H(z_0)$ obtained as the
  intersection of each ball $B_{i}(t\s)$ with~$H(z_0)$.  For any
  angle~$\theta \in [\theta_0 - \eps, \theta_0 + \eps]$, consider the
  projection of the disks on the orthogonal complement~$\ello$ of a
  line with orientation~$\theta$.  Each disk projects on an interval.
  The intersection of the projections of $B_{i}(t\s)$ and $B_{j}(t\s)$
  is non-empty, since $\theta \in \K(B_{i}(t\s)B_{j}(t\s), z_{0})$.
  By Helly's theorem in one dimension, this implies that the
  intersection of all intervals is not empty.  Therefore there exists
  a line transversal to the disks with orientation~$\theta$, and by
  construction it meets the disks in order.  Since this holds for any
  angle~$\theta$ in this interval, we have a transversal intersecting
  the interior of each disk by Lemma~\ref{lem:two-not-pinned}.
\end{proof}

\section{Semi-algebraic reformulation of
  Conjectures~\ref{conj:abcd-acdb} and~\ref{conj:double-pinning}}
\label{sec:semialgebraic}

Counter-examples to Conjectures~\ref{conj:abcd-acdb}
and~\ref{conj:double-pinning} can be expressed as solutions of sets of
polynomial equalities and inequalities, therefore reducing these
conjectures to the question of the emptiness of a semi-algebraic
set. Various algorithms are known to answer this question and several
implementations are available~\cite{bpr}, so in
principle settling our conjecture is only a matter of computational
resources. The resources needed to solve a given problem can be
greatly influenced by the modeling of the problem. We therefore
believe that although our attempts in this direction failed, there is
value in summarizing our efforts.

\paragraph{Tangency condition.}

Our first formulation yields a semi-algebraic set defined in $\R^{10}$
by four equalities (two of of degree $6$, two of degree $10$) and
twelve quadratic inequalities. It describes the configurations of four
non-overlapping unit balls $\{A, B, C, D\}$ and two common tangents to
these balls in the geometric permutations $\p ABCD.$ and $\p ACDB.$.
It follows from Lemma~\ref{lem:shrink-conj} and
Theorem~\ref{the:double-pinning} that the existence of such a
configuration is equivalent to falsifying
Conjectures~\ref{conj:abcd-acdb} and~\ref{conj:double-pinning}.

Up to translation and symmetry we can assume that $a$ is at the
origin, $b$ is on the $x$-axis and $c$ is on the $xy$-plane. The four
points can thus be described using six variables $x_b , x_c , y_c ,
x_d , y_d , z_d$.  We next argue that parameterizing the directions of
the two lines, rather than the lines themselves, is
sufficient. Indeed, the geometric permutation realized by a line
tangent or transversal to the four balls can be read from its
direction vector $\vec{v}$ alone: that line meets $X$ before $Y$ if
and only if $\vec{v} \cdot \vc{xy} > 0$.  Also, we can assume that the
centers of the balls are not coplanar, since in this case
Conjecture~\ref{conj:abcd-acdb} is known to hold, so Equation~(6) of
Borcea et al.~\cite{Borcea06} allows to retrieve the full description
of the line from its direction and the coordinates of the centers of
the balls.

So let $\vc{v_1}, \vc{v_2}$ denote the direction vectors of two common
tangents in order, respectively, $\p ABCD.$ and $\p ACDB.$.  Since no
line parallel to the $yz$-plane can be a common tangent to the balls
$A$ and $B$, up to scaling we can write the two vectors $\vc{v_1} =
(1,q,r)$ and $\vc{v_2} = (1,s,t)$. The condition that $\vc{v_i}$ is a
direction of a common tangent to the four balls is equivalent to
Equations~(7) and~(8) of Borcea et al.~\cite{Borcea06}; we thus have
for each $\vc{v_i}$ two equalities, one of degree~$6$ with $27$ terms,
the other of degree~$10$ with $195$ terms. We then require that the
balls be disjoint by adding six quadratic inequalities that require
that the squared distance between any two centers is at least $4$. We
finally ensure that the two lines meet the balls in the right order by
six bilinear inequalities that constrain the signs of $\vc{v_1}\cdot
\vc{ab}$, $\vc{v_1}\cdot \vc{bc}$, $\vc{v_1}\cdot \vc{cd}$ and
likewise for $\vc{v_2}$.

\paragraph{Pinning conditions.}

Our second formulation builds on Conjecture~\ref{conj:double-pinning}
and yields a semi-algebraic set defined, essentially, in $\R^{10}$ by
six equalities of degree $4$, six linear inequalities and six
inequalities of degree~$4$.

We first describe an alternating hyperboloidal configuration
$(\F,\sigma)$ using $5$ variables $(h,t_a, t_b, t_c, t_d) \in \R^5$,
six degree-four inequalities and three linear inequalities.
Specifically, we equip $\R^3$ with an orthonormal frame such that
$\sigma$ is the $x$-axis and its minimal pinning by $\F$ is witnessed
by the hyperbolic paraboloid with equation $xy=-hz$; the centers of the
balls are thus on the quadric with equation $xy=hz$.  The position of
ball $W$ is given by the variable $t_w$ that represents the
$x$-coordinate of the tangency point of $W$ and $\sigma$. Assuming the
balls have unit radius, the position of the center is then:
\[ w = \pth{\frac{2ht_w}{1-t_w^2}, \frac{1-t_w^2}{1+t_w^2}, \frac{2t_w}{1+t_w^2}}\]
It remains to require that the balls be non-overlapping (six
inequalities of degree four bounding from below the squared distances
between centers) and to specify the order in which the balls touch the
line $\sigma$ (three linear inequalities ordering the $t_w$'s).

Our semi-algebraic set then describes the existence of two
configurations of the previous type that realize the geometric
permutations $\p ABCD.$ and $\p ACDB.$ and where the tetrahedra of
centers are isometric.\footnote{Two tetrahedra with equal edge-lengths
  are either isometric or one is isometric to a reflection of the
  other. A point satisfying all these constraints may thus not
  correspond to a pair of isometric minimal-pinning
  configurations. Yet, mirroring one of the configurations would give
  a pair of isometric minimal pinnings so as far as we only care about
  existence, this system is fine as~is.} Indeed, if two such
configurations exist then a rigid motion that maps each ball from the
first configuration to the matching ball in the second configuration
will send the $x$-axis of the first configuration to a transversal
realizing $\p ABCD.$ in the second configuration. Conversely, if a
counter-example to Conjectures~\ref{conj:double-pinning} exists then
it gives rise to a pair of configuration as described above.

We build our system by picking two independent sets of variables
$(h,t_a, t_b, t_c, t_d)$ and $(h',t_a', t_b', t_c', t_d')$, collecting
the linear inequalities enforcing the orders $\p ABCD.$ on one
configuration and $\p ACDB.$ on the other, collecting the six
degree-four inqualities enforcing that the first configuration is
non-overlapping (we drop their counterparts in the second
configuration as they are redundant if the configurations are
isometric) and adding the condition $|uv|^2 = |u'v'|^2$ for each of
the six pairs in $\{a,b,c,d\}$. 

We expect the variety defined by these equations (ie. dropping the
inequations) to have dimension~$4$: the lower bound comes from the
system ($10$ variables minus $6$ equations, hoping for transverse
intersection) and the upper bound comes from the geometry (dimension
$5$ or more would imply that in the space of configurations of four
balls with two tangents, the set of configurations where both tangents
are pinned has codimension~$1$ whereas we expect that when deforming
such a configuration, both pinnings need not happen simultaneously).
The system contains a $5$-dimensional degenerate component that
corresponds to $hh' = 0$, as the parameterization then degenerates. We
therefore add one variable $u$ and the equality $u*h*h' = 1$ so that the
existence of a solution in $u$ forces the other term to be
non-zero.\footnote{This trick of enforcing an inequality $I \neq 0$ by
  adding a variable $u$ and an equality $uI=1$ is known as
  ``saturation''.}

\paragraph{Discussion.}

The first system seems a challenge for symbolic methods such as
critical point techniques. Indeed, such methods typically first
operate on the underlying variety, obtained by dropping from the
system all inequalities. This variety therefore corresponds to
configurations of four, possibly intersecting, unit balls with two
common tangents, in any order. That variety appears to be unmixed,
that is, it contains components of different dimensions, and to
contain a singular locus of dimension at least~$2$ (those
configurations where a tangent is pinned). Already an equidimensional
decomposition of this system seems out of reach at the moment. The
second system seems better suited but also seems out of reach at the
moment.
\end{document}